\documentclass[a4paper,12pt]{amsart}
\usepackage[utf8]{inputenc}
\usepackage[T1]{fontenc}
\usepackage{fullpage}
\usepackage[english]{babel}
\usepackage{mathrsfs}
\usepackage{pgfplots}
\usepackage{amsmath, amsthm, amssymb}
\usepackage[retainorgcmds]{IEEEtrantools} 
\title{Hardy-Orlicz Spaces of conformal densities}
\author{Sita Benedict}
\newtheorem{theorem}{Theorem}[section]
\newtheorem{lemma}[theorem]{Lemma}

\newtheorem{corollary}[theorem]{Corollary}

\newtheorem*{ghthm}{Gehring-Hayman Theorem}
\newtheorem*{theorem*}{Theorem}
\newtheorem*{proposition*}{Proposition}
\newtheorem*{corollary*}{Corollary}
\newtheorem*{remark*}{Remark}
\numberwithin{equation}{section}
\newcommand{\C}{\mathbb{C}} 
\newcommand{\Bn}{\mathbb{B}^{n}}
\newcommand{\Bt}{\mathbb{B}^{2}}
\newcommand{\Rn}{\mathbb{R}^{n}}

\newcommand{\Sn}{\mathbb{S}^{n-1}}

\newcommand{\omn}{\omega_{n-1}}
\begin{document}
\begin{abstract}
We define and prove characterizations of Hardy-Orlicz spaces of conformal densities. 
\end{abstract}

\footnotetext{
{\it 2010 Mathematics Subject Classification: 30C35, 30H10}\\
{\it Key words and phrases. Hardy spaces, Hardy-Orlicz, conformal densities}
\endgraf The author was partially supported by the Academy of Finland grants
131477 and 263850.}

\maketitle
\section{introduction}

In \cite{myfirstpaper} the authors defined a new type of Hardy-Orlicz space by considering the internal path distance from $f(x)$ to $f(0)$ in place of the euclidean distance $|f(x)|$, where $f$ is a conformal mapping of $\mathbb{B}^2$. The internal distance between two points $f(x), f(y) \in f(\mathbb{B}^2)$ is formally
\begin{eqnarray*}
d_I(f(x), f(y)) = \inf_\gamma \int_\gamma |f'|ds,
\end{eqnarray*}
where the infimum is taken over all curves in $\mathbb{B}^2$ with endpoints $x$ and $y$. Thus $d_I$ is a metric on $f(\mathbb{B}^2)$, but it can equivalently be thought of as a metric on $\mathbb{B}^2$ associated with the conformal mapping $f$. Since the definition depends on $|f'|$ and not on $f$ we can think of $|f'|$ as a special kind of \textit{density} on $\mathbb{B}^2$ and ask what are the properties of $|f'|$ that are actually needed to develop the theory of intrinsic Hardy-Orlicz spaces. In general, a density on $\Bn$ is simply a Borel measurable function $\rho:\Bn \rightarrow [0,\infty]$. For a given density $\rho$ the $\rho$-length of a curve $\gamma$ in $\Bn$  is
\begin{align*}
\textnormal{length}_\rho(\gamma) = \int_\gamma \rho(x) ds,
\end{align*}
where $ds$ denotes integration with respect to arc length. If $\rho$ is continuous and strictly positive we can define the metric $d_\rho$ on $\Bn$ by setting
\begin{eqnarray*}
d_\rho(x,y) = \inf \textnormal{length}_\rho(\gamma), \;\;x,y\in\Bn,
\end{eqnarray*}
where the infimum is taken over all curves $\gamma \subset \Bn$ with endpoints $x$ and $y$.

It was shown in \cite{conformalmetrics} that a continuous density $\rho:\Bn \rightarrow (0,\infty)$ need only satisfy two simple conditions so that a number of classical results from geometric function theory in the plane generalize to the setting of conformal densities on $\Bn$. The first is a Harnack-type inequality (HI), where $\rho$ does not variate much on Whitney-type scales in the ball. We say that $\rho$ satisfies HI($A)$ if there exists a constant $A \geq 1$ such that for all $z\in \Bn$
\begin{align*}
1/A \leq \frac{\rho(x)}{\rho(y)} \leq A, \textnormal{whenever} \;x,y\in B_z = B(z, 1-|z|/2).
\end{align*}
The second condition is a Volume Growth condition (VG). To state this we associate with $\rho$ a Borel measure $\mu_\rho$ on $\Bn$ by setting
\begin{eqnarray*}
\mu_\rho(E) = \int_E \rho^n dx
\end{eqnarray*}
for each Borel set $E \subseteq \Bn$, and we say that $\rho$ satisfies condition VG$(B)$ if there exists a constant $B > 0$ such that
\begin{eqnarray*}
\mu_\rho(B_\rho(x,r)) \leq Br^n
\end{eqnarray*}
for all $x\in\Bn$ and $r>0$. A \textit{conformal density} on $\Bn$ is then any continuous density $\rho:\Bn \rightarrow (0,\infty)$ satisfying both HI$(A)$ and VG$(B)$. 

It is simple to show that $|f'|$ is a conformal density with constants $A = e^{12}$ and $B=\pi$ whenever $f$ is conformal mapping of $\Bt$. If $g:\Bn \rightarrow \Rn$ is quasiconformal, then it can be shown that the averaged derivative of $g$, usually denoted as $a_g$, is a conformal density on $\Bn$. For details and other examples of conformal densities, including ones that do not arise from a quasiconformal mapping, see \cite{conformalmetrics}.

We show in this paper that the same two conditions HI($A$) and VG($B$) are sufficient to develop a Hardy-Orlicz space theory for conformal densities on the unit ball in $\Rn$. 

Let $\psi: [0, \infty] \rightarrow [0, \infty]$ be a strictly increasing, differentiable function with $\psi(0) = 0$, or \textit{growth function} for short. We say that a conformal density $\rho$ on $\Bn$ belongs to the Hardy-Orlicz space $H^\psi$ if there is a $\delta > 0$ such that 
\begin{eqnarray*}
\sup_{0 < r < 1}\int_{\Sn}\psi(\delta |(r\omega)|_\rho)d\sigma < \infty,
\end{eqnarray*} 
where the distance $|(r\omega)|_\rho$ between $r\omega$ and 0 is the one induced by $\rho$, and $\sigma$ is the n-1-dimensional surface measure on $\Sn$. If $\psi(t) = t^p$ for a given $0 < p < \infty$ then we simply denote the corresponding Hardy space with the symbol $H^p$. Our first result gives several characterizations of these spaces that hold for all growth functions $\psi$.

\begin{theorem}\label{mainthm}
Let $\rho$ be a conformal density on $\Bn$ and $\psi$ a growth function. Then the following are equivalent:
\begin{enumerate}
\item $\rho \in H^\psi$
\item $\psi(\delta_1 |\omega|_\rho) \in L^1(\Sn)$ \textnormal{for some} $\delta_1 > 0$
\item $\psi(\delta_2 \rho^*(\omega)) \in L^1(\Sn)$ \textnormal{for some} $\delta_2 > 0$
\item $(1-r)^{n-2}\psi(\delta_3 M(r,\rho)) \in L^1((0,1))$ \textnormal{for some} $\delta_3 > 0$.
\end{enumerate}
\end{theorem}

The definitions of the non-tangential maximal function $\rho^*$ and the maximum modulus $M(r,\rho)$ are given in Section 3. The characterizations in Theorem \ref{mainthm} are analogues to results that hold for the classical $H^p$ spaces of analytic functions on the unit disk, which follow from results in \cite{nevan}, \cite{prawitz} and \cite{hardylittlewoodineq}. The classical characterization involving the maximum modulus holds only when restricting to conformal mappings of $\Bt$ belonging to $H^p$. These results have also been proved in the more general setting of quasiconformal mappings on $\Bn$, see \cite{hpqc}. For more on the theory of the classical $H^p$ spaces see for example \cite{durenhp}. 

Other characterizations that hold for classical $H^p$ when restricting to the conformal mappings have been established in more recent years. For instance, it was established in both \cite{hpqc} and \cite{dirichlet} that if $f$ is a conformal mapping of $\Bt$ then
\begin{eqnarray}\label{classicalresult}
f\in H^p\;\;\textnormal{if and only if} \;\;\int_{\Bt} |f'(x)|^p(1-|x|)^{p-1}dx < \infty
\end{eqnarray}
for every $0 < p < \infty$. We show, as a corollary to statements proved for more general growth functions, that the corresponding statement for conformal densities is also true, see Theorem \ref{energychar} below. Our result, in combination with a theorem from \cite{myfirstpaper} which says that the internal Hardy spaces and classical Hardy space classes of conformal mappings are the same for all $0< p < \infty$, also implies the result (\ref{classicalresult}), and so our work gives an alternative, and shorter, proof to the ones found in \cite{hpqc} and \cite{dirichlet}. See Section 4.
\begin{theorem}\label{energychar}
Let $\rho$ be a conformal density of $\Bn$ and $0 < p < \infty$. Then
\begin{eqnarray*}
\rho \in H^p\;\;\textnormal{if and only if} \;\;\int_{\Bn} \rho(x)^p(1-|x|)^{p-1}dx < \infty.
\end{eqnarray*}
\end{theorem}
It is well known that a conformal map $f$ on the unit disk belongs to the classical $H^p$ space for all $p < 1/2$.  As a consequence of Theorem \ref{mainthm} every conformal density also belongs to $H^p$ for all $p$ in a certain range. 
\begin{theorem}\label{everydensity}
There exists a constant $p_0 = p_0(n, A, B) > 0$ so that every conformal density $\rho: \Bn \rightarrow (0, \infty)$ satisfying \textnormal{HI}$(A)$ and \textnormal{VG}$(B)$ belongs to $H^p$ for all $p < p_0$.
\end{theorem}
We obtain as a corollary by way of the Gehring-Hayman theorem (see Section 2) the following. 
\begin{corollary}There exists a constant $p_0 = p_0(n, A, B) > 0$ so that 
\begin{eqnarray*}
\int_{\Sn}\left(\int_0^1 \rho(t\omega)dt\right)^pd\sigma< \infty
\end{eqnarray*}
whenever $0 < p < p_0$ and $\rho$ is a conformal density satisfying \textnormal{HI}$(A)$ and \textnormal{VG}$(B)$.
\end{corollary}
This paper is organized as follows. Section 2 covers notation, modulus of curve families and also the Gehring-Hayman Theorem. In Section 3 we prove Theorem \ref{mainthm}, and in Section 4 we prove our results that require an additional assumption on $\psi$.
\section{Preliminaries}
We set $\Bn = \{x\in \Rn : |x| < 1 \}$ and $\Sn = \{x\in \Rn : |x| = 1\}$, and in general $B(x,r)$ denotes the open ball in $\Rn$ centered at $x$ and with radius $r > 0$. For each $x\in\Bn$ let 
\begin{eqnarray*}B_x = B(x,(1-|x|)/2)
\end{eqnarray*} 
and 
\begin{eqnarray*}
S_x = \left\{\frac{x}{|x|} : x\in B_x \right\} \subseteq \Sn,
\end{eqnarray*} 
and for each $\omega \in \Sn$ let
\begin{eqnarray*}
\Gamma(\omega) = \bigcup\{B_{t\omega} : 0 \leq t < 1\}
\end{eqnarray*}
be the Stolz cone centered at $\omega$. The surface area of $\Sn$ will be denoted as $\omega_{n-1}$.

Whenever we write a constant as $C = C(A,B,...)$ we mean that the constant depends only on the values $A,B,...$. In a proof the value of a constant can change from one line to the next without any notational indication or explanation. We will write $A \approx B$ to indicate that there exists a constant $C$ such that
\begin{align*}
\frac{A}{C} \leq B \leq CA.
\end{align*}

Let $\rho$ be a conformal density on $\Bn$ and $d_\rho$ the metric on $\Bn$ induced by $\rho$. For each $x\in\Bn$ we abbreviate
\begin{eqnarray*}
|x|_\rho = d_\rho(x,0).
\end{eqnarray*}
The metric extends to the boundary in the sense that
\begin{eqnarray*}
d_\rho(\omega, x) = \inf \textnormal{length}_\rho(\gamma)
\end{eqnarray*}
is well defined for each $\omega\in\Sn$ and  $x \in \Bn$ by taking the infimum over all curves $\gamma$ in $\Bn$ with endpoints $\omega$ and $x$. We abbreviate also
\begin{eqnarray*}
|\omega|_\rho = d_\rho(\omega,0).
\end{eqnarray*}
The subscript $\rho$ will be used to denote the usual metric notions in the metric space $(\Bn, d_\rho)$. For example, given $x\in\Bn$ and $r>0$ we set $B_\rho(x,r) = \{y\in\Bn : d_\rho(x,y) < r\}$.

One of our main tools is the modulus of curve families, defined here. Let $\Gamma$ be a family of locally rectifiable curves in $\Bn$. The modulus Mod$\Gamma\in [0,\infty]$ is defined to be
\begin{eqnarray*}
\textnormal{Mod}\Gamma = \inf_\varrho \int_{\Bn} \varrho^n dx,
\end{eqnarray*}
where the infimum is taken over all Borel measurable functions $\varrho:\Bn \rightarrow [0,\infty]$ that satisfy length$_\varrho(\gamma)\geq1$ for every $\gamma \in \Gamma$. For certain families of curves, the exact value of the modulus is easy to calculate. For instance, if $E$ is a Borel set in $\Sn$ and $\Gamma$ is the collection of radial segments with one endpoint in $B(0,r), 0 < r < 1$, and the other endpoint in $E$ then
\begin{eqnarray*}
\textnormal{Mod}\Gamma = \sigma(E)(\log(1/r))^{1-n}.
\end{eqnarray*}
See \cite{vaisala} for this result and other properties of the modulus. 

We will need the following modulus estimate from \cite[Lemma 3.2]{conformalmetrics}.
\begin{lemma}\label{moduluslemma}
Let $\rho$ be a conformal density on $\Bn$ satisfying \textnormal{VG}$(B)$. Then there exists a constant $C(B, n) \geq \omega_{n-1}$ with the following property. Let $E$ be a non-empty subset of $\Bn$ and suppose $L \geq \delta > 0$. Assume that \textnormal{diam}$_\rho(E) \leq \delta$ and that $\Gamma$ is a family of curves in $\Bn$ so that $\gamma$ has one endpoint in $E$ and length$_\rho(\gamma) \geq L$ for every $\gamma \in \Gamma$. Then
\begin{eqnarray*}
\textnormal{mod}\Gamma \leq \frac{C}{[\log(1 + L/\delta)]^{n-1}}.
\end{eqnarray*}
\end{lemma}
Now using simple modulus techniques we obtain the following. 
\begin{lemma}\label{moduluslemma2}
Let $\rho$ be a conformal density. There exists a constant $C = C(n, A, B)$ such that 
\begin{eqnarray*}
\sigma(\{\omega\in S_x : d_\rho(w,x) > M\rho(x)(1-|x|)\}) \leq C\sigma(S_x)(\log M)^{1-n}
\end{eqnarray*}
for any $x\in \Bn$ and $M > 1$. 
\end{lemma}
\begin{proof}
Let $x\in\Bn$ and $E = \{\omega\in S_x : d_\rho(w,x) > M\rho(x)(1-|x|)\}$. Suppose first that $|x| < 1/4$. If $\Gamma_E$ is the collection of radial segments with one endpoint in $E$ and the other in $B_x  \cap S(0,1/4)$ then Mod$(\Gamma_E) = \sigma(E)(\log 4)^{1-n}$. By property HI(A) and the definition of the set $E$ there is a constant $C = C(n,A)$ such that each curve in $\Gamma_E$ has one endpoint in $B_\rho(x, C\rho(x)(1-|x|))$ and the other in $\overline{\Bn}\setminus B_\rho(x, M\rho(x)(1-|x|))$. If $2 \leq C$ and $C^2 < M$ then
\begin{eqnarray*}
\sigma(E)(\log 4)^{1-n} = \textnormal{Mod}(\Gamma_E) \leq C(\log M)^{1-n} \leq C\omn(\log M)^{1-n}
\end{eqnarray*}
by Lemma \ref{moduluslemma}. If $1 < M \leq C^2$ then trivially
\begin{eqnarray*}
\sigma(E) \leq \omn(\log C^2)^{n-1}(\log M)^{1-n}.
\end{eqnarray*}

If $1/4 \leq |x|$ and $\Gamma_E$ is the collection of radial segments with one endpoint in $E$ and the other endpoint in $B_x \cap S(0, |x|)$ then Mod$(\Gamma_E) = \sigma(E)(\log \frac{1}{|x|})^{1-n}$. Like before, Lemma \ref{moduluslemma} implies that
\begin{eqnarray*}
\sigma(E)(\log 1/|x|)^{1-n} = \textnormal{Mod}(\Gamma_E) \leq C(\log M)^{1-n} 
\end{eqnarray*}
whenever $2 \leq C$ and $C^2 < M$. The other case is again trivial, so noting that $(\log 1/|x|)^{n-1} \approx \sigma(S_x)$ we are done. 
\end{proof}

The following version of the Gehring-Hayman theorem is a generalization of a result originally proved by Gehring and Hayman in \cite{GH}. This version was proved in \cite{conformalmetrics} using the modulus of curve families as a primary tool. Recall that for all $x\in\Bn$ the hyperbolic geodesic from connecting $0$ and $x$ is the radial segment $[0,x]$.  
\begin{ghthm}
Let $\rho$ be a conformal density on $\Bn$. There is a constant $C(A, B, n)$ with the following property. If $\gamma$ is a hyperbolic geodesic in $\Bn$ with endpoints in $\overline{\Bn}$
 and $\tilde{\gamma}$ is any other curve in $\Bn$ with the same endpoints, then
 \begin{eqnarray*}
 \textnormal{length}_\rho(\gamma) \leq C \textnormal{length}_\rho(\tilde{\gamma}).
 \end{eqnarray*}
 \end{ghthm}

\section{Proof of characterization theorem}
For each conformal density $\rho$ define the non-tangential maximal function $\rho^*$ on $\Sn$ as
\begin{align*}
\rho^*(\omega) = \sup_{x\in \Gamma(\omega)}|x|_\rho.
\end{align*}

\begin{lemma}\label{nontanglemma}
Let $\psi$ be a growth function, $\rho$ a conformal density and $\delta > 0$. There exists a constant $C = C(A,B,n)$ such that
\begin{align*}
\int_{\Sn} \psi(\frac{\delta}{C}\rho^*(\omega))d\sigma \leq \int_{\Sn}\psi(\delta|\omega|_\rho) d\sigma. 
\end{align*}
\end{lemma}
\begin{proof}
Let $\omega \in \Sn$ and $x\in \Gamma(\omega)$. Then $x\in B_{t\omega}$ for some $0 < t < 1$. The Gehring-Hayman theorem and HI($A$) imply that
\begin{align*}
|x|_\rho \leq |t\omega|_\rho + d_\rho(x, t\omega) \leq C\text{length}_\rho([0, \omega)) \leq C|\omega|_\rho,
\end{align*}
from which the result easily follows.
\end{proof}
A measure $\mu$ on $\Bn$ is called a \textit{Carleson measure} if there exists a constant $C(\mu)>0$ such that
\begin{align*}
\mu(\Bn \cap B(\omega, r)) \leq C(\mu) r^{n-1}
\end{align*}
for all $\omega \in \Sn$ and all $r > 0$. We denote the infimum of all such constants $C(\mu)$ by $\alpha_\mu$. 

\begin{lemma}\label{ghapp}
Let $\rho$ be a conformal density, $\psi$ a growth function, $\delta > 0$ and $\mu$ a Carleson measure on $\Bn$. There are constants $C_1 = C_1(A,B,n)$ and $C_2 = C_2(n, \alpha_\mu)$ such that
\begin{align*}
\int_{\Bn} \psi(\delta/C_1 |x|_\rho)d\mu \leq C_2\int_{\Sn}\psi(\delta |\omega|_\rho)d\sigma.
\end{align*}
\end{lemma}
\begin{proof}
Let $\epsilon > 0$ and set $E(\lambda) = \{x\in \Bn : \epsilon|x|_\rho > \lambda\}$ and $U(\lambda) = \{\omega \in \Sn : \epsilon\rho^*(\omega) > \lambda \}$ for each $\lambda > 0$. We can use the generalized form of the Whitney decomposition \cite[Theorem III.1.3]{whitneyref} to write the open set $U(\lambda)$ as
\begin{align*}
U(\lambda) = \bigcup_{k=1}^\infty  S_{x_k},
\end{align*}
where the points $x_k\in\Bn$ are chosen so that each $\omega \in U(\lambda)$ belongs to no more than $N(n)$ caps  $S_{x_k}$ and also so that $(1 - |x_k|)/C \leq d(S_{x_k}, \partial U(\lambda)) \leq C(1 - |x_k|)$. The constant is absolute and the distance is the spherical distance on $\Sn$. It follows by the properties of the Whitney decomposition that $E(\lambda) \subset \bigcup_{k=1}^\infty B(x_k/|x_k|, C(1 - |x_k|))$, for some absolute constant $C$. Then
\begin{eqnarray*}
\mu(E(\lambda)) &\leq& \sum_{k=1}^\infty \mu(B(x_k/|x_k|, C(1 - |x_k|))\cap\Bn)\\ &\leq& C(n, \alpha_\mu) \sum_{k=1}^\infty (1-|x_k|)^{n-1} \\ &\leq& C(n, \alpha_\mu) \sum_{k=1}^\infty \sigma(S_{x_k}) \leq C(n, \alpha_\mu)\sigma(U(\lambda)).
\end{eqnarray*}
Then,
\begin{eqnarray*}
\int_{\Bn} \psi (\epsilon|(x)|_\rho) d\mu&=& \int_0^\infty \psi'(\lambda) \mu(E(\lambda))d\lambda \\ &\leq& C(n, \alpha_\mu) \int_0^\infty  \psi'(\lambda) \sigma(U(\lambda)) d\lambda \\&=& C(n, \alpha_\mu)\int_{\Sn} \psi(\epsilon \rho^*(\omega))d\sigma.
\end{eqnarray*}
By applying Lemma \ref{nontanglemma} with an appropriate choice of $\epsilon$ we are done.
\end{proof}
With each conformal density $\rho$ we associate the maximum modulus function 
\begin{align*}
M(r, \rho) = \sup_{|x|\leq r} |x|_\rho.
\end{align*}
defined for $r \in [0, 1)$. We define the function over the closed ball rather than the sphere of radius $r$ so that the function is strictly increasing. By the Gehring-Hayman theorem there is a constant $C(A,B,n)$ such that
\begin{align}\label{maxmodequiv}
\sup_{|x| = r}|x|_\rho \leq M(r,\rho) \leq C\sup_{|x| = r}|x|_\rho.
\end{align}
\begin{proof}[Proof of Theorem \ref{mainthm}]
Lemma \ref{nontanglemma} shows that (2) implies (3). By definition (3) implies both (2) and (1). If $C$ is the constant from the Gehring-Hayman theorem then Fatou's lemma and the Gehring-Hayman theorem imply that 
\begin{eqnarray*}
\int_{\Sn} \psi(\delta/C |\omega|_\rho) d\sigma &\leq& \int_{\Sn} \psi(\delta/C \textnormal{length}_\rho([0,\omega))) d\sigma \\ &=& \int_{\Sn} \liminf_{r\rightarrow 1}\psi(\delta/C \textnormal{length}_\rho([0,r\omega])) d\sigma \\ &\leq& \liminf_{r\rightarrow 1}\int_{\Sn} \psi(\delta/C \textnormal{length}_\rho([0,r\omega])) d\sigma \\ &\leq& \sup_{0 < r < 1} \int_{\Sn} \psi(\delta |r\omega|_\rho) d\sigma,
\end{eqnarray*}
which shows that (1) implies (2). Thus (1), (2), and (3) are equivalent, and we now proceed to show the equivalence of (2) and (4). 

First assume (4). We will show that there is a constant $C > 0$ such that 
\begin{eqnarray*}
\int_{\Sn} \psi\left(\frac{\delta}{3}|\omega|_\rho\right) d\sigma \leq C\int_0^1 (1-t)^{n-2} \psi(\delta M(r,\rho)) dr.
\end{eqnarray*}
We start by rewriting the integral on the left as
\begin{eqnarray}\label{cavform1}
\int_{\Sn} \psi\left(\frac{\delta}{3}|\omega|_\rho\right) d\sigma = \int_0^\infty \psi'(\lambda)\sigma(\{\omega \in \Sn: \frac{\delta}{3}|\omega|_\rho > \lambda \})d\lambda.
\end{eqnarray}
Let $E = \{\omega \in \Sn: \frac{\delta}{3}|\omega|_\rho > \lambda \}$ for a fixed $\lambda$. We will obtain an upper bound on $\sigma(E)$ using modulus of curve families. Indeed, there exists a unique $r = r(\lambda)$ such that 
\begin{eqnarray*}
r_\lambda = \sup \{r\in [0, 1) : \delta M(r, \rho) = \lambda \}.
\end{eqnarray*}
Denote by $\Gamma_E$ the path family consisting of the radial segments connecting $B(0,r_\lambda)$ to $E$. Then,
\begin{eqnarray*}
\textnormal{Mod}(\Gamma_E) = \frac{\sigma(E)}{(\log(1/r_\lambda))^{n-1}} \geq \frac{\sigma(E)}{2(1-r_\lambda)^{n-1}}
\end{eqnarray*}
when $1/2 < r_\lambda < 1$. Since each curve in $\Gamma_E$ has one endpoint belonging to $B_\rho(0, \lambda/\delta)$ and one endpoint in $\overline{\Bn}\setminus B_\rho(0, 3\lambda/\delta)$, we can apply Lemma \ref{moduluslemma} to obtain a constant $C = C(B,n)$ such that 
\begin{eqnarray*}
\sigma(E) \leq C(1-r_\lambda)^{n-1}
\end{eqnarray*}
whenever $1/2 < r_\lambda < 1$ and $r_\lambda$ is defined as above. If $\nu$ is the measure on $[0,1]$ defined by $d\nu = (1-t)^{n-2}dt$ then
\begin{eqnarray*}
\nu(\{t\in [0,1]: M(t,\rho) > \lambda/\delta\}) = (1-r_\lambda)^{n-1}/(n-1).
\end{eqnarray*}
This estimate and Fubini's theorem applied to the right hand side of (\ref{cavform1}) give
\begin{gather*}
\int_{\Sn} \psi\left(\frac{\delta}{3}|\omega|_\rho\right) d\sigma \\\leq \omega_{n-1}\psi(\delta M(1/2,\rho)) + C \int_{\delta M(1/2,\rho)}^\infty \psi'(\lambda)\nu(\{t\in [0,1]: M(t,\rho) > \lambda/\delta\})d\lambda   \\\leq \omega_{n-1}\psi(\delta M(1/2,\rho)) + C \int_0^\infty \psi'(\lambda)\int_{\{t\in [0,1]: M(t,\rho) > \lambda/\delta\}} (1-t)^{n-2}dt d\lambda \\ \leq C\int_0^1 (1-t)^{n-2}\psi(\delta M(t,\rho))dt,
\end{gather*}
which is what we needed to show. 

Conversely, assume (2) holds for some $\delta > 0$, and choose points $x_k\in \Bn$ such that ${|x_k| = r_k = 1 - 2^{-k}}$ and $|x_k|_\rho = \sup_{|x| = r_k}|x|_\rho, k=1,2,\mathellipsis$. Given any $\epsilon > 0$ we have
\begin{eqnarray*}
\int_0^1 (1-r)^{n-2}\psi(\epsilon M(r,\rho))dr &\leq& 2^n\sum_{k=1}^\infty (2^{-k})^{n-1}\psi(\epsilon M(r_k, \rho)) \\ &\leq& 2^n\sum_{k=1}^\infty (2^{-k})^{n-1}\psi(C \epsilon |x_k|_\rho) \\ &=& 2^n\int_{\Bn} \psi(C\epsilon |x|_\rho)d\mu,
\end{eqnarray*} 
where $d\mu(x) = \sum_{k=1}^\infty (1-|x|)^{n-1}\delta_{x_k}$ and $C$ is the constant from (\ref{maxmodequiv}). The measure $\mu$ is a Carleson measure, and so by Lemma \ref{ghapp} there are universal constants $C_1$ and $C_2$ such that
\begin{eqnarray*}
\int_{\Bn} \psi(C\epsilon |x|_\rho)d\mu \leq C_1\int_{\Sn}\psi(C_2\epsilon |\omega|_\rho)d\sigma.
\end{eqnarray*}
The proof is finished by letting $\epsilon = \delta/C_2$.
\end{proof}
\section{Characterizations under additional conditions on $\psi$}
A growth function is \textit{doubling} if there exists a constant $C$ such that $\psi(2t) \leq C\psi(t)$ for all $t\in[0,\infty].$ The infimum of all such constants is called the \textit{doubling constant} of $\psi$ and is denoted by $C_\psi$. 

\begin{lemma}\label{doubling}
Let $\psi$ be a doubling growth function and $\rho$ a conformal density. If 
\begin{eqnarray*}
\int_{\Bn} \psi(\rho(x)(1-|x|))\frac{dx}{1-|x|} < \infty
\end{eqnarray*}
then $\psi(|\omega|_\rho) \in L^1(\Sn)$.
\end{lemma}
\begin{proof}
Our first step is to show that the given assumptions imply that 
\begin{eqnarray}\label{finiteintegral}
\int_{\Sn}\psi(v(\omega))d\sigma < \infty,
\end{eqnarray} 
where $v(\omega) = \sup_{x\in \Gamma(\omega)}(\rho(x)(1-|x|))$. To that end, fix $\omega \in \Sn$ and let $x\in \Gamma(\omega)$. Then there is a constant $C = C(A,n, C_\psi)$
such that
\begin{eqnarray*}
\psi(\rho(x)(1-|x|)) &\leq&\frac{C}{(1-|x|)^n}\int\limits_{B_x} \psi(\rho(y)(1-|y|))dy \\&\leq& C\int\limits_{\Gamma(\omega)} \frac{\psi(\rho(y)(1-|y|))}{(1-|y|)^n}dy.
\end{eqnarray*}
Thus,
\begin{eqnarray*}
\int_{\Sn}\psi(v(\omega))d\sigma \leq C\int_{\Sn}\int\limits_{\Gamma(\omega)} \frac{\psi(\rho(y)(1-|y|))}{(1-|y|)^n}dy,
\end{eqnarray*}
and so it is enough to show that the integral on the right is finite.
If $u(y) = \frac{\psi(\rho(y)(1-|y|))} {1-|y|}$, then by the assumptions $u$ is integrable on $\Bn$ and Fubini's Theorem gives
\begin{eqnarray*}
\int_{\Sn} \int_{\Gamma(\omega)} u(y) (1-|y|)^{1-n} dy d\sigma &=& \int_{\Bn} u(y)(1-|y|)^{1-n}\int_{\Sn}\chi_{\Gamma(\omega)}(y) d\sigma dy \\&\approx& \int_{\Bn} u(y) dy < \infty,
\end{eqnarray*}
which completes the first step.

We now use (\ref{finiteintegral}) to show that $\psi(|\omega|_\rho) \in L^1(\Sn)$. Let $U(\lambda) = \{\omega \in \Sn : \rho^*(\omega) > \lambda \}$ for each $\lambda > 0$. Since $U(\lambda)$ is an open set we can use the generalized form of the Whitney decomposition to express $U(\lambda)$ as a union of caps $S_{x_j}$
\begin{eqnarray*}
U(\lambda) = \bigcup S_{x_j},
\end{eqnarray*}
where the caps have uniformly bounded overlap and 
\begin{eqnarray}\label{whitdist}
(1-|x_j|)/C \leq d(S_{x_j}, \partial U(\lambda)) \leq C (1-|x_j|).
\end{eqnarray}

If $\omega \in S_{x_j}$ and $v(\omega) \leq \gamma$ then by (\ref{whitdist}) and property HI(A) there exists $\omega' \in \Sn \setminus U(\lambda)$ and a corresponding $x_j' \in \Gamma(\omega')$ such that
\begin{eqnarray*}
|x_j|_\rho \leq d_\rho(x_j, x_j') + |x_j'|_\rho \leq C\gamma + \lambda.
\end{eqnarray*}
Now let $M >1$ and $\gamma= \frac{\lambda}{(M+1)C}$ and suppose $\omega \in S_{x_j}$ with $v(\omega) \leq \gamma$ and  $|\omega|_\rho > 2\lambda$. Then, by what we showed above and the definition of $v(\omega)$,
\begin{eqnarray*}
d_\rho(\omega, x_j) \geq |\omega|_\rho - |x_j|_\rho > \lambda - C\gamma = MC\gamma \geq M\rho(x_j)(1-|x_j|),
\end{eqnarray*}
and therefore
\begin{gather*}
\sigma(\{\omega\in S_{x_j}: |\omega|_\rho > 2\lambda\; \textnormal{and}\; v(\omega) \leq \gamma \}) \\ \leq
\sigma(\{\omega\in S_{x_j}: d_\rho(\omega, x_j)  > M\rho(x_j)(1-|x_j|) \}) \\
\leq C\sigma (S_{x_j}) (\log M)^{1-n}
\end{gather*}
by Lemma \ref{moduluslemma2}. If $|\omega|_\rho > 2\lambda$ then $\omega \in U(\lambda)$, and so by the above we have
\begin{gather*}
\sigma(\{\omega \in \Sn: |\omega|_\rho > 2\lambda \} \\ \leq \sigma (\{\omega \in U(\lambda): |\omega|_\rho > 2\lambda \; \textnormal{and}\; v(\omega) \leq \gamma \}) + \sigma (\{\omega \in \Sn: v(\omega) > \gamma \}) \\ \leq C\sum_j \sigma (S_{x_j}) (\log M)^{1-n} + \sigma (\{\omega \in \Sn: v(\omega) > \gamma \}) \\ \leq C\sigma (U(\lambda))(\log M)^{1-n} + \sigma (\{\omega \in \Sn: v(\omega) > \gamma \}).
\end{gather*} 
Thus
\begin{eqnarray*}
\int_{\Sn} \psi(\frac{1}{2}|\omega|_\rho) d\sigma = \int_0^\infty \psi'(\lambda) \sigma(\{\omega \in \Sn: |\omega|_\rho > 2\lambda \})d\lambda \\ \leq \int_0^\infty \psi'(\lambda) \left(C\sigma (U(\lambda))(\log M)^{1-n} + \sigma \left(\left\{\omega \in \Sn: v(\omega) > \frac{\lambda}{(M+1)C} \right\}\right)\right)d\lambda \\ = C(\log M)^{1-n}\int_{\Sn}\psi(\rho^*(\omega))d\sigma + \int_{\Sn} \psi((M+1)C v(\omega)) d\sigma \\ \leq C(n,A)(\log M)^{1-n}\int_{\Sn}\psi(\rho^*(\omega))d\sigma  + C(M, n,A, C_\psi)\int_{\Sn} \psi(v(\omega))d\sigma,
\end{eqnarray*}
where $C_\psi$ is the doubling constant of $\psi$. We would like to use Lemma \ref{nontanglemma} to bring the integral involving $\rho^*$ to the left side of the inequality, but since both integrals could be infinite we first apply Lemma \ref{nontanglemma} to the above for the conformal densities $\rho_t(x) = \rho(tx)$, $0 < t < 1$. By choosing $M$ large enough and taking the limit as $t \rightarrow 1$ we obtain
\begin{eqnarray*}
\int_{\Sn} \psi(|\omega|_\rho) d\sigma \leq C(M, n,A, C_\psi)\int_{\Sn} \psi(v(\omega))d\sigma,
\end{eqnarray*}
which completes the proof.
\end{proof}
A full converse to Lemma \ref{doubling} is not possible, as the following example shows. Let $p(x) \equiv 1$ on $\Bn$. Then $\rho$ is clearly a conformal density and also $\psi(|\omega|_\rho) \in L^1(\Sn)$ for every growth function $\psi$. If
\begin{eqnarray}\label{example}
\psi(t) =
\begin{cases}
\frac{1}{\log\frac{1}{t}}, & t<1/2 \\
\frac{2t}{\log2}, & t \geq 1/2,
\end{cases}
\end{eqnarray}
then $\psi$ is a growth function that is doubling, while
\begin{eqnarray*}
\int_{\Bn} \psi(\rho(x)(1-|x|)) \frac{dx}{1-|x|} = C + C\int_{1/2}^1\frac{1}{(1-r)\log\frac{1}{1-r}}dr
\end{eqnarray*}
is infinite. Assuming superadditivity or concavity plus an additional growth restriction on $\psi(t)$ near $t=0$ we obtain converses in the following forms.
\begin{lemma}\label{convex}
Let $\psi$ be a growth function such that $\psi(t_1) + \psi(t_2) \leq \psi(t_1 + t_2)$ for all $t_1,t_2 \in [0, \infty]$. If $\rho \in H^\psi$ then there is $\delta>0$ such that
\begin{eqnarray*}
\int_{\Bn} \psi(\delta\rho(x)(1-|x|))\frac{dx}{1-|x|} < \infty.
\end{eqnarray*}
\end{lemma}
\begin{proof}
Let $\delta > 0$. By switching to polar coordinates, applying property HI$(A)$, the superadditivity of $\psi$ and the Gehring-Hayman theorem we have that
\begin{eqnarray*}
\int_{\Bn} \psi(\delta\rho(x)(1-|x|))\frac{dx}{1-|x|} &\leq& \int_{\Sn}\int_0^1\frac{\psi(\delta\rho(t\omega)(1-t))}{1-t} dtd\sigma\\  &\leq& \int_{\Sn}C_1\sum_j \psi(\delta C_2\rho(t_j\omega)(1-t_j)) d\sigma \\  &\leq& \int_{\Sn}C_1\psi\left(\sum_j \delta C_2\rho(t_j\omega)(1-t_j)\right) \\  &\leq& \int_{\Sn}C_1\psi\left(\delta C_3\int_0^1 \rho(t\omega)dt\right)d\sigma \\  &\leq& \int_{\Sn}C_1\psi(\delta C_4 |\omega|_\rho)d\sigma.
\end{eqnarray*}
The last integral is finite for an appropriately chosen $\delta$ by Theorem \ref{mainthm}.
\end{proof}
Note that the growth function from (\ref{example}) does not satisfy the multiplicative assumption in the next lemma.
\begin{lemma}\label{concave}
Let $\psi$ be a growth function that is concave and for which there exists $C > 0$ such that $\psi(ab) \leq b\psi(Ca)$ whenever $a\geq 0$ and $0 < b < 1$.  Then if $\rho \in H^\psi$ then
\begin{eqnarray*}
\int_{\Bn} \psi(\rho(x)(1-|x|))\frac{dx}{1-|x|} < \infty.
\end{eqnarray*}
\end{lemma}
\begin{proof}
By Jensen's inequality and the multiplicative property of $\psi$ to get that 
\begin{eqnarray*}
\psi^{-1}\left(\int_0^1 \frac{\psi(\rho(t\omega)(1-t))}{1-t} dt \right)  &=& \lim_{r\rightarrow1}\psi^{-1}\left(\int_0^r \frac{\psi(\rho(t\omega)(1-t))}{1-t} dt \right)\\&\leq& \int_0^1 C\rho(t\omega)dt.
\end{eqnarray*}
Then by the Gehring-Hayman theorem 
\begin{eqnarray*}
\int_{\Bn} \psi(\rho(x)(1-|x|))\frac{dx}{1-|x|} &\leq& \int_{\Sn}\int_0^1\frac{\psi(\rho(t\omega)(1-t))}{1-t} dtd\sigma\\  &\leq& \int_{\Sn}\psi\left(C\int_0^1 \rho(t\omega)dt\right)d\sigma  \\  &\leq& \int_{\Sn}\psi( C |\omega|_\rho)d\sigma,
\end{eqnarray*}
which is finite by Theorem \ref{mainthm}. 
\end{proof}
\begin{proof}[Proof of Theorem \ref{energychar}]
This follows clearly from Lemmas \ref{doubling}, \ref{convex} and \ref{concave}.
\end{proof}
In \cite{myfirstpaper} the authors proved that if $f:\Bt\rightarrow\C$ is conformal, then $f$ belongs to the classical Hardy space $H^p$ if and only if 
\begin{eqnarray*}
\sup_{0<r<1} \int_{\mathbb{S}^1}|f(r\omega)|_I^pd\sigma < \infty
\end{eqnarray*}
for all $0 < p < \infty$, which in terms of our notation here is equivalent with saying that the conformal density $|f'|$ belongs to $H^p$. Thus the result from the classical setting stated in (\ref{classicalresult}) follows also by Theorem \ref{energychar}. The main tools needed in the result from \cite{myfirstpaper} and in the proof of Theorem \ref{energychar} are modulus of curve families and the use of properties of conformal densities, and so this new proof of (\ref{classicalresult}) is shorter and less technical than those in \cite{hpqc} and \cite{dirichlet}. The earlier proofs relied on, for example, the use of Carleson measures in the case of \cite{hpqc} and on the use of several older theorems including that from Pommerenke \cite{gerpomm}, Hayman \cite{multivalent} and two Hardy-Littlewood inequalities in the case of \cite{dirichlet} We note that especially in one direction, the proofs in the setting of conformal densities are very straightforward.
\begin{proof}[Proof of Theorem \ref{everydensity}]
Let $\rho$ be a conformal density on $\Bn$. By \cite[Theorem 5.1]{conformalmetrics} there exist constants $\beta(B,n) > 1, C_1(A,B,n), C_2(A,B,n)$ so that
\begin{eqnarray*}
C_1(1-|x|)^{\beta - 1} \leq \frac{\rho(x)}{\rho(0)}\leq C_2 \frac{1}{(1-|x|)^{\beta - 1}}
\end{eqnarray*}
for every $x\in\Bn$. Since by the Gehring-Hayman theorem 
\begin{eqnarray*}
M(r,\rho) \approx \sup_{\omega \in \Sn}\int_0^r \rho(t\omega)dt
\end{eqnarray*}
for all $0 < r < 1$, there exists some constant $C >0$ such that
\begin{eqnarray*}
M(r,\rho) \leq C (1-r)^{-\beta}.
\end{eqnarray*}
Thus
\begin{eqnarray*}
\int_0^1 (1-r)^{n-2}M(r,\rho)^p dr
\end{eqnarray*}
is finite for any $p < \frac{n-1}{\beta}$ and $\rho \in H^p$ for all such $p$ by Theorem \ref{mainthm}.
\end{proof}

 \noindent\textit{Acknowledgement.} This paper forms a part of the thesis of the author, written under the supervision of Pekka Koskela. 

\bibliographystyle{plain}
\bibliography{secondpaperref}

\noindent  Sita Benedict\\
Department of Mathematics and Statistics\\
University of Jyv\"{a}skyl\"{a}\\
P.O. Box 35 (MaD)\\
FI-40014\\
Finland\\

\noindent{\it E-mail address}: \texttt{sita.c.benedict@jyu.fi}

\end{document}